\documentclass[11pt]{article}

\usepackage{pstricks}
\usepackage{amssymb}
\usepackage{epsfig}
\usepackage[mathscr]{eucal}
\usepackage{colortbl}
\usepackage{multirow}

\newcommand{\R}{\mathbb{R}}

\newcommand{\N}{\mathbb{N}}

\renewcommand{\epsilon}{\varepsilon}
\newcommand{\beq}{\begin{equation}}
\newcommand{\eeq}{\end{equation}}
\newcommand{\bea}{\begin{eqnarray}}
\newcommand{\eea}{\end{eqnarray}}
\newcommand{\bean}{\begin{eqnarray*}}
\newcommand{\eean}{\end{eqnarray*}}

\newtheorem{theorem}{Theorem}
\newtheorem{lemma}[theorem]{Lemma}

\newtheorem{proposition}[theorem]{Proposition}
\newtheorem{definition}[theorem]{Definition}

\newenvironment{proof}%
{\par\noindent{\em Proof.\ }}%
{\ \hfill ~\rule{2mm}{2mm}\par\medskip}
{\par\noindent{\bf Nota:\ }}%
{\ \hfill \par\medskip}
\newenvironment{remark}%
{\par\noindent{\bf Remark.\ }}%
{\ \hfill \par\medskip}

\def\calf{\mathcal{F}}

\def\calo{\mathcal{O}}

\def\qh{\mathcal{P}}
\def\calc{\mathcal{C}}
\def\cald{\mathcal{D}}

\def\qh{\mathscr{P}}
\def\QH{\mathcal{Q}}

\def\Lequiv{\mathcal{L}}

\def\Opl{{\ell}}

\def\hgx{\widehat{g}}
\def\gx{g}
\def\tgx{\tilde{g}}
\def\hx{h}
\def\zero{{\mathbf 0}}

\def\x{{\mathbf x}}
\def\X{{\mathbf X}}

\def\p{{\mathbf p}}

\def\D{{\mathbf D}}
\def\F{{\mathbf F}}
\def\G{{\mathbf G}}

\def\P{{\mathbf P}}

\def\t{{\mathbf t}}

\def\fp{{\mathrm{f}}}
\def\co{{\mathrm{c}}}
\def\di{{\mathrm{d}}}

\def\diverg{{\mathrm{div}}}
\def\sig{{\mathrm{sig}}}
\def\det{{\mathrm{det}}}

\def\dim{{\mathrm{dim}}}
\def\Cor{{\mathrm{Cor}}}
\def\Proy{{\mathrm{Proy}}}
\def\Ker{{\mathrm{Ker}}}
\def\Range{{\mathrm{Range}}}
\def\FZERO{{{\mathbf D}_0}}

\def\bracket#1{\left[ #1 \right]}
\def\parent#1{\left( #1 \right)}

\def\llave#1{\left\{ #1 \right\}}

\def\vedo#1#2{\left(\begin{array}{c} #1 \\ #2 \end{array}\right)}

\def\fracp#1#2{{\textstyle{\frac{#1}{#2}}}}
\def\matriz#1{\parent{\begin{array} #1 \end{array} } }
\def\matriz22#1#2#3#4{\parent{\begin{array}{cc} #1&#2\\#3&#4 \end{array} } }
\def\Res#1#2{\mathrm{Res}\left[#1,#2\right]}
\def\ano{a}
\def\bno{b}
\def\bmo{c}

\begin{document}

\title{\bf Analytic integrability around a nilpotent singularity}

\author{Antonio Algaba, Crist\'obal  Garc\'{\i}a\\
\small Dept.  Matem\'aticas, Facultad de Ciencias,  Univ. of Huelva, Spain. \\
\small {\rm e--mails:} {\tt algaba@uhu.es}, {\tt cristoba@uhu.es}\\
\vspace{0.01cm}\\
Jaume Gin\'e \\
\small Departament de Matem\`atica, Inspires Research Centre,\\
\small Universitat de Lleida, Av. Jaume II, 69, 25001, Lleida, Catalonia, Spain.\\
\small {\rm e--mail:} {\tt gine@matematica.udl.cat}}
\date{}

\maketitle

\abstract{In this work it is solved the analytic integrability problem around a nilpotent singularity of a differential system in the plane under generic conditions.}

\section{Introduction and statement of the main result}

One of the most important problems in the study of a planar differential system
\bea  \label{sys1}
\dot{x}= P(x,y), \qquad \dot{y} = Q(x,y),
\eea
where $P$ and $Q$ are analytic  in a neighborhood of the origin and coprimes, is to determine when it
has a local analytic first integral defined in a neighborhood of a singular point.
Other important problem is to characterize when a singular point of system (\ref{sys1})
is a center. It is well--known that system (\ref{sys1}) has a center at a singular point only if it is monodromic and it has either linear part of center type, i.e., with imaginary eigenvalues (nondegenerate singular point), or nilpotent linear part (nilpotent singular point) or null linear part (degenerate singular point). A nondegenerate singular point is a center if and only if it has an analytic first integral around the singular point, see \cite{Lyapunov66,Poincare}. However the analytic integrability does not characterize the nilpotent or degenerate centers, see for instance \cite{AGG,CGGL,GGL,Gine05,GGL2} although some nilpotent or degenerate centers have an analytic first integral, see for instance \cite{AGG,GGL}.

\smallskip

In this work, we focus on the study of the analytic integrability for differential systems in the plane.
For differential systems with non-null linear part we have the following cases in function their eigenvalues:
if $\lambda_1 \lambda_2 \ne 0$ we have either a saddle, or node or a center type singular point. If $\lambda_1=0$ and $\lambda_2 \ne 0$ we have a saddle-node.  Finally if $\lambda_1=\lambda_2=0$ we have a nilpotent singular point.
The nodes and saddle-nodes are not analytically integrable. Based on the results commented before we know that a singular point with linear part of center type is analytically integrable if, and only if, it is a center. Moreover it is a center if, and only if, it is orbitally linearizable, see \cite{AGG,GaMa,Poincare2}.
For a saddle singular point, i.e., $\lambda_1< 0 <\lambda_2$, if $\lambda_1/\lambda_2 \not \in \mathbb{Q}$ then system (\ref{sys1}) is not analytically integrable around the singular point. If $\lambda_1/\lambda_2= -p/q \in \mathbb{Q}$ then we have a {\it resonant saddle}. A resonant saddle has an analytic first integral around the singular point if, and only if, it is orbitally linearizable, see for instance \cite{FSZ,HJ,ZY,Zoladek} and references therein. The most studied resonant saddles are resonant saddles of Lotka-Volterra systems, see \cite{CMR,CR,GR,LCC,LCL}. The local analytic integrability problem for some particular differential systems in the plane and for $n$-dimensional differential systems is studied in \cite{DRZ,LV,RXZ,Z}.

\smallskip

We now introduce some notation in order to present the main results of this work.
A scalar polynomial $f$ is quasi-homogeneous of type $\t=(t_1,t_2)\in\mathbb{N}^2$ and degree $k$ if
$f(\epsilon^{t_1}x,\epsilon^{t_2}y)=\epsilon^{k}f(x,y)$.
The vector space of quasi-homogeneous scalar polynomials of type $\t$ and degree $k$ is denoted by $\qh_{k}^\t$.
A polynomial vector field $\F=(P,Q)^T$ is quasi-homogeneous of type $\t$ and degree $k$ if $P\in\qh_{k+t_1}^\t$ and $Q\in\qh_{k+t_2}^\t$. The vector space of polynomial quasi-homogeneous vector fields of type $\t$ and degree $k$ is denoted by $\QH_{k}^\t$. Given an analytic vector field $\F$, we can write it as a quasi-homogeneous expansion corresponding to a fixed type $\t$:
\begin{equation}\label{eq2}
\F(\x)=\F_r(\x)+\F_{r+1}(\x)+\cdots=\sum_{j\geq r}\F_j ,
\end{equation}
where $\x\in\R^2$, $r\in\mathbb{Z}$ and $\F_j\in\QH_j^\t$ i.e., each term $\F_j$ is a quasi-homogeneous vector field of type $\t$ and degree $j$. Any
$\F_j\in\QH_j^{\t}$ can be uniquely written as
\beq\label{componentes}
\F_j =   \X_{\hx_j} + \mu_{j}\FZERO,
\eeq
where $ \mu_{j}=\fracp{1}{r+|\t|} \, \diverg\parent{\F_{j}}\in\qh_{j}^{\t}$,
$\hx_{j}=\fracp{1}{r+|\t|}\FZERO\wedge\F_{j} \in\qh_{j+|\t|}^{\t}$,
$\FZERO=(t_1x,t_2y)^T$, and $\X_{\hx_{j}}=\parent{-\partial \hx_{j} / \partial y ,
\partial \hx_{j}/\partial x}^T$ is the  Hamiltonian vector
field with Hamiltonian function $\hx_{j}$ (see \cite[Prop.2.7]{AlgabaNonlinearity09}).

\medskip

Notice that the condition of polynomial integrability of the first quasi-homogeneous component is a necessary condition in order to be $\F$ analytically integrable, that is:
\begin{lemma}\label{leintFintFr} Let $\F=\sum_{j\geq r}\F_j$ a vector field, $\F_j\in\QH_j^\t$. If $\F$ is analytically integrable then $\F_r$ is polynomially integrable.
\end{lemma}
The main results of this work are the following.

\begin{theorem}\label{main}
Let $\F=\sum_{j\geq r}\F_j$, $\F_j\in\QH_j^\t$ be a nilpotent vector field such that the origin of $\dot{\x}=\F_r(\x)$ is isolated and $\F_r$ is polynomially integrable, then $\F$ is analytically integrable if, and only if, it is orbitally equivalent to $\F_r$.
\end{theorem}

\begin{remark}
The condition that the origin of $\dot{\x}=\F_r(\x)$ is an isolated equilibrium is generic, in fact the only remaining case without studying is $\F=(y,bx^ny)^T+\cdots$, $b\neq 0$, see Proposition \ref{propreFN}. The characterization of the integrability in this case is an open issue that will be addressed in a future work.
\end{remark}
The following theorem characterizes the analytic integrability of a nilpotent vector field through the existence of a Lie symmetry.

\begin{theorem}\label{main2}
Let $\F=\sum_{j\geq r} \F_j$, $\F_j\in\QH_j^\t$, be a nilpotent vector field such that the origin of $\dot{\x}=\F_r(\x)$ is isolated and $\F_r$ is polynomially integrable, then $\F$ is analytically integrable if, and only if, there exist a vector field $\G=\sum_{j\geq 0}\G_j$, $\G_j\in\QH_j^\t$, $\G_0=\FZERO$ and a scalar function $\mu$, $\mu(\zero)=r$ such that $[\F,\G]=\mu\F$, hence $\F$ has a Lie symmetry.
\end{theorem}

These results apply to nilpotent singular points. More specifically, here we solve the nilpotent case, that is, with $h(0,y)=-y^2/2$, where $\hx$ is the hamiltonian function of its first quasi-homogeneous component. Similar results for nondegenerate and resonant saddle singular points were found in previous works, see \cite{AGG,HJ,Zoladek}.

\smallskip

Theorem \ref{main} is not true for degenerate singular points, that is, any analytic integrable
vector field with null linear part is not orbitally equivalent to its first component, see for instance \cite{ACG}.
Hence the remaining global open case is the case when we have a degenerate singular point. However some partial results are known. The analytic integrability problem when $\F_r=\X_{\hx}$, i.e., the first quasi-homogeneous component of $\F$ is conservative, with $h$ having only simple factors is completely solved in \cite{AlgabaNonlinearity09}. The case when $\F_r=\X_{\hx}$ with $h$ having multiple factors is still open.
In \cite{Algaba_Checa_Garcia_Gine2015} it is studied a particular case of the case when $\F_r=\X_{\hx}$, with $h$ having multiple factors showing the difficulty of this problem.

The case when $\F_r=\X_{\hx}+\mu \FZERO$, $\mu \ne 0$, with $h$ having simple or multiple factors is also open.

Throughout the paper we will not consider questions of convergence in the normal forms because the formal integrability is equivalent to the analytical integrability for vector fields analyzed, see \cite{Mattei80}.

\section{Takens-Bogdanov singularity}

In order to study the analytic integrability of the nilpotent differential systems we analyze the normal preform which sets the Newton diagram of the field and therefore the type and the degree of the quasi-homogeneous leading term of the vector field that we want to study.

\begin{proposition}[Normal preform]\label{propreFN} Consider $\dot{\x}=\F(\x)$ a nilpotent vector field with $\zero$ an isolated singular point. There exist $n\in\mathbb{N}$, a polynomial change $\Phi$ and a type $\t$ such that $\tilde{\F}:=\Phi_*\F=\tilde{\F}_r+\cdots$, $\tilde{\F}_r\in\QH_r^t$, with $\cdots$ the quasi-homogeneous terms of type $\t$ and degree greater than $r$, where $\tilde{\F}_r$ satisfies one of the following two conditions
\begin{description}
\item [A)] $\tilde{\F}_r=(y,bx^ny)^T\in\QH_n^\t$, $b\in\mathbb{R}$, $b\neq 0$, i.e., $\t=(1,n+1)$, and $r=n$.
\item[B)] $\zero$ is an isolated singular point of $\dot{\x}=\tilde{\F}_r(\x)$ and there exist $\Psi_0\in\QH_0^\t$, with $\det\parent{D\Psi_0(\zero)}\neq 0$ such that $\G_r=(\Psi_0)_*\tilde{\F}_r$ is one of the following vector fields.
\begin{description}
\item[B1)] $\G_r=\vedo{y}{x^{2n}}\in\QH_{2n-1}^{(2,2n+1)}$, i.e.,  $\t=(2,2n+1)$ and $r=2n-1$.
\item[B2)] $\G_r=\vedo{y}{-(n+1)x^{2n+1}}+dx^{n}\vedo{x}{(n+1)y}\in\QH_{n}^{(1,n+1)}$, $d\in\mathbb{R}$, i.e., $\t=(1,n+1)$ and  $r=n$.
\item[B3)] $\G_r=\vedo{y}{0}+dx^{n}\vedo{x}{(n+1)y}\in\QH_{n}^{(1,n+1)}$, $d\in\mathbb{R}$, $d\neq 0$  i.e., $\t=(1,n+1)$ and $r=n$.
\item[B4)] $\G_r=\vedo{y}{(n+1)x^{2n+1}}+dx^{n}\vedo{x}{(n+1)y}\in\QH_{n}^{(1,n+1)}$, $d\in\mathbb{R}$, $d\neq 1$, i.e., $\t=(1,n+1)$ and $r=n$.
\end{description}
\end{description}
\end{proposition}
\begin{proof}
 A nilpotent singularity can be written as
\begin{equation}\label{que}
\begin{array}{lll}
\dot{x} &= & y + x f_1(x) + y f(x,y),\\
\dot{y} &= & g_1(x) + y g_2(x) + y^2 g(x,y).
\end{array}
\end{equation}
Let us denote { by $M$ the lowest-degree in the Taylor} expansion of
$g_1(x)$; and  $N$ is the minimum of the {lowest}-degrees of the
Taylor expansions for $f_1(x)$ and $g_2(x)$. Hence, $M=\infty$
arises if $g_1(x)\equiv 0$ and $N=\infty$ corresponds to
$f_1(x)\equiv g_2(x)\equiv 0$. Then, we can write nilpotent system (\ref{que}) as
 \bean
\dot{x} & = & y + x^{N+1} \tilde{f}_1(x) + y f(x,y),\\
\dot{y} & = & x^{N} y \tilde{g}_1(x) + x^{M} \tilde{g}_2(x)+ y^2 g(x,y),
 \eean
where $M\in\N\cup\llave{\infty}$, $M>1$  $N\in\N\cup\llave{\infty}$,
and $\tilde{f}_1(x)=\ano+\calo\parent{x}$, $\tilde{g}_1(x) =
\bno+\calo\parent{x}$, $\tilde{g}_2(x) = \bmo+\calo\parent{x}$,  $f(x,y)= \calo\parent{x,y}$,
$g(x,y)  = \calo\parent{x,y}$.
\begin{itemize}
\item If $M=N=\infty$, the line $y=0$ is a curve of singular points, the origin is not isolated and we must exclude this case.
\item If $M=2n<2N+1$ then $\Phi=Id$ and the first component of the nilpotent vector field respect to the type $\t=(2,2n+1)$ is $\tilde{\F}_r=(y,cx^{2n})^T$, $c\neq 0$ of degree $r=2n-1$. $\zero$ is an isolated singular point of $\dot{\x}=\tilde{\F}_r(\x)$ and taking $\Psi_0(x,y)=(c^{\fracp{1}{2n+1}}x,c^{\fracp{1}{2n+1}}y)^T$, we get $\G_r=(y,x^{2n})^T$. This corresponds to the case  {\bf B1)}.
\item If $M=2n+1\leq 2N+1$ or $2N+1< M$, we divide the study in the following cases:
\begin{itemize}
\item If $M=2n+1<2N+1$ then the first component of the nilpotent vector field respect to the type $\t=(1,n+1)$ is $\F_r=(y,cx^{2n+1})^T$, of degree $r=n$, with $c\neq 0$.
\item If $M=2n+1=2N+1<+\infty$ then the first component of the nilpotent vector field respect to the type $\t=(1,n+1)$ is $\F_r=(y+ax^{n+1},bx^ny+cx^{2n+1})^T$, of degree $r=n$, with $c(a^2+b^2)\neq 0$.
\item If $2N+1<M$ then the first component of the nilpotent vector field respect to the type $\t=(1,n+1)$ is $\F_r=(y+ax^{n+1},bx^ny)^T$, of degree $r=n$, with $(a^2+b^2)\neq 0$.
\end{itemize}
We can deal with these three cases at once by writing
$$
 \F_{r}(x,y)=\vedo{y+ax^{n+1}}{bx^{n}y+ c x^{2n+1}},
 $$
where $c=0$, $a^2+b^2\neq 0$ if $2N+1<M$, $N=n$, $c(a^2+b^2)\neq 0$ if $M=2n+1=2N+1$.

$\zero$ is an isolated singular point of $\dot{\x}=\F_r(\x)$ if, and only if, $c\neq ab$.
\begin{itemize}
\item If $c=ab$ taking $\Phi(x,y)=(x,y+ax^{n+1})$ we get $\tilde{\F}_n=(y,(b+(n+1)a)x^ny)^T$.
\begin{itemize}
\item If $2N+1<M$, $c=0$ and $a^2+b^2\neq 0$. Therefore $b+(n+1)a\neq 0$. This correspond to the case {\bf A)}.
\item If $M=2n+1=2N+1$ and $b+(n+1)a\neq 0$ this correspond to the case  {\bf A)}.
\item If $M=2n+1=2N+1$ and $b+(n+1)a=0$ we have a nilpotent system with new values of $N$ and $M$.
\end{itemize}
\item If $c\neq ab$, $\zero$ is an isolated singular point of $\dot{\x}=\F_r(\x)$. It is a simple task to perform the splitting (\ref{componentes}) in
this case. We obtain $r=n$, $r+|\t|=2(n+1)$, and
 \bean
\hx(x,y)&=&\fracp{c}{2(n+1)} x^{2(n+1)} +
\left(\fracp{{\bno}}{2(n+1)}-\fracp{1}{2}{\ano}\right)x^{n+1}y
-\fracp{1}{2}y^2\\
&=&-\fracp{1}{2}\parent{y-\parent{\fracp{{\bno}}{2(n+1)}
-\fracp{{\ano}}{2}}x^{n+1}}^2 + \frac{\Big(
c + \fracp{\parent{\bno-\ano(n+1)}^2}{4(n+1)}
 \Big)}{2(n+1)} x^{2(n+1)},
\\
\mu(x,y)&=&\left(\fracp{{\bno}}{2(n+1)}+\fracp{1}{2}{\ano}\right)x^{n}.
 \eean
\begin{itemize}
\item If $c + \fracp{\parent{\bno-\ano(n+1)}^2}{4(n+1)}=0$, taking $\Psi_0(x,y)=\left(x,y-\parent{\fracp{{\bno}}{2(n+1)}-
\fracp{{\ano}}{2}}x^{n+1}\right)^T$ we get $\parent{\Psi_0}_*\F_r=(y+dx^{n+1},(n+1)dx^ny)^T$ with $d=\fracp{b+(n+1)a}{2(n+1)}\neq 0$, otherwise $c=(n+1)a^2=ab$ and this is contradictory. This correspond to the case {\bf B3)}.
\item If $(n+1)A:=c + \fracp{\parent{\bno-\ano(n+1)}^2}{4(n+1)}\neq 0$, taking $$\Psi_0(x,y)=\left(|A|^{\frac{1}{2n}}x,|A|^{\frac{1}{2n}}\left(y-\parent{\fracp{{\bno}}{2(n+1)}-
\fracp{{\ano}}{2}}\right)x^{n+1}\right)^T,$$ we get $\parent{\Psi_0}_*\F_r=(y+dx^{n+1},\sigma x^{2n+1}+(n+1)dx^ny)^T$ with $\sigma=\sig(A)=\pm 1$ and $d=\fracp{b+(n+1)a}{2(n+1)\sqrt{|A|}}$.
\begin{itemize}
\item $\sigma=-1$ is the case {\bf B2)}.
\item If $\sigma=1$ then $d\neq 1$. Otherwise $c=ab$ and this is contradictory. This correspond to the case {\bf B4)}.
\end{itemize}
\end{itemize}
\end{itemize}
\end{itemize}
\end{proof}

\begin{remark}
The case {\bf A)} will not be treated in this work. In this case its first quasi-homogeneous component is reducible. It is an open issue that will be studied in a future work.
\end{remark}

\subsection{Invariant curves of nilpotent vector fields}

We need the three following results in order to calculate the formal invariant curves of a nilpotent vector field whose first quasi-homogeneous component is non-conservative.

\smallskip

First we give the definition of invariant curve and its associated cofactor.

\begin{definition} Let $f\in\mathbb{C}[[x,y]]$
( power series with coefficients in $\mathbb{C}$). It is said that
$f$, $f(\zero)=0$, is an invariant curve of $\F$ (or
$\dot{\x}=\F(\x)$) if there exists $K\in\mathbb{C}[[x,y]]$
(cofactor) such that $\nabla f(\x)\cdot\F(\x)=K(\x)f(\x)$, for all
$\x\in\mathbb{C}^2$.
\end{definition}

\begin{lemma}\label{lemm}
Consider $\F_n =   \X_{\hx} + \mu \FZERO \in\QH_n^\t$ where $\hx\in\mathbb{C}[x,y]$ (polynomials with coefficients in $\mathbb{C}$). Any factor of $h$ is an invariant curve of $\F_n$.
Moreover any quasi-homogeneous irreducible invariant curve of $\F_n$ is a factor of $h$.
\end{lemma}

\begin{proof} Let $f\in\qh_s^\t$ an irreducible factor of $\hx$ then $\hx=fg$ and $\nabla f\cdot\F_n=\nabla f\cdot X_{fg}+\mu\nabla f\cdot\FZERO=f\nabla f\cdot X_g+s\mu f=\parent{\nabla f\cdot X_g +s\mu}f$. Therefore, $f$ is an invariant curve of $\F_n$.

If $f\in\qh_s^\t$ is an irreducible invariant curve of $\F_n$ with cofactor $K$ then $k f=\nabla f\cdot\F_n=\nabla f\cdot X_{\hx}+\mu\nabla f\cdot\FZERO=\nabla f\cdot X_{\hx}+ s\mu f$. Therefore, $\nabla f\cdot X_{\hx}=(K-s\mu)f$ and $f$ is an irreducible invariant curve of $\X_{\hx}$. So, $f$ divides to $\hx$.
\end{proof}

\begin{lemma}\label{leqrangprang} Consider $\F_n=(y,(n+1)x^{2n+1})^T+dx^n\FZERO\in\QH_n^\t$ with $n\in\mathbb{N}$, $\t=(1,n+1)$, $\FZERO(x,y)=(x,(n+1)y)^T\in\QH_0^\t$.
Let $f\in\qh_{s}^\t$ be an irreducible invariant curve of $\F_n$ with cofactor $K_n \in \qh_n^\t$. Consider the linear operators:
\bean
\begin{array}{rcl}\delta_{n+k}:\qh_{n+k-s}^\t&\longrightarrow&\qh_{n+k}^\t\\ p&&pf\end{array}&&\begin{array}{rcl}\tilde{\Opl}_{n+k}:\qh_{k}^\t&\longrightarrow&\qh_{n+k}^\t\\ p&&\nabla p\cdot\parent{\F_n-\fracp{1}{k}K_n\FZERO}.\end{array}
\eean
If $|d|\neq 1+\fracp{2(n+1)}{k-n-1}$ then $\tilde{\Opl}_{n+k}\parent{\Cor\parent{\delta_{k}}}\cap\Range\parent{\delta_{n+k}}=\llave{0}$,
where $\Cor\parent{\delta_k}$ is a complementary subspace to $\Range\parent{\delta_k}$.
\end{lemma}
\begin{proof}
Consider $k=(n+1)k_1+k_2$ with $0\leq k_2<n+1$. By Lemma \ref{lemm} the unique irreducible invariant curves of $\F_n$ are the factors of $\hx=-(y^2 - x^{2n+2})/2$, given by $f=y-x^{n+1}$ and $f=y+x^{n+1}$ whose cofactors are $K_n=(n+1)(d-1)x^n$ and $K_n=(n+1)(d+1)x^n$, respectively.

Any polynomial $p\in\qh_k^\t$ can be expressed as $p(x,y)=\sum_{j=0}^{k_1} a_jx^{k-(n+1)j}f^j$, hence we can assume that $\Cor\parent{\delta_{n+k}}=<x^{k}>$. We consider $p=a_0x^{k}\in\Cor\parent{\delta_k}$. If we denote $\G_n=\parent{\F_n-\fracp{K_n}{k}\FZERO}$ we have to prove that if $\nabla p\cdot\G_n\in\Range\parent{\delta_{n+k}}$ then $p=0$.

We prove the case $f=y-x^{n+1}$, the case $f=y+x^{n+1}$ is analogous.
\bean \nabla p\cdot\G_n&=&a_0kx^{k-1}\nabla x\cdot\G_n=a_0kx^{k-1}\parent{y+\fracp{(k-n-1)d+n+1}{k}x^{n+1}}\\
&=&a_0k x^{k-1} \left ( y-x^{n+1}+\fracp{(k-n-1)d+k+n+1}{k}x^{n+1} \right). \eean
If $d\neq -\fracp{k+n+1}{k-n-1}=-1-\fracp{2(n+1)}{k-n-1}$, $\nabla x\cdot\G_n$ is not a multiple of $y-x^{n+1}$. 	
In order that $\nabla p\cdot\G_n$ be a multiple of $y-x^{n+1}$ we must take $a_0=0$ and therefore $p=0$.
\end{proof}

\begin{lemma}\label{leqhsumadirecta} Consider $k,s\in\mathbb{N}$, $k>s$, $s\geq t_1t_2$, $\F_n\in\QH_n^\t$. Let $f\in\qh_{s}^\t$ be an invariant curve of $\F_n$ with cofactor $K_n$, $f$ different from $x=0$ and $y=0$. Consider the linear operators defined
in Lemma \ref{leqrangprang}.
If $\Cor\parent{\delta_{k}}$ is a complementary subspace of $\Range\parent{\delta_{k}}$ and  $\tilde{\Opl}_{n+k}\parent{\Cor\parent{\delta_{k}}}\cap\Range\parent{\delta_{n+k}}$ $=\llave{0}$, then given $q\in\qh_{n+k}^\t$ there exist a unique $p^{(1)}\in\Cor\parent{\delta_{k}}$ and a unique $p^{(2)}\in\qh_{n+k-s}^\t$ such that
$$q=\tilde{\Opl}_{n+k}\parent{p^{(1)}}+p^{(2)}f.$$
\end{lemma}
\begin{proof} We divide the proof in some steps.
\begin{itemize}
\item[i)] Firstly we prove that $\Cor\parent{\delta_{k}}\cap\Ker\parent{\tilde{\Opl}_{n+k}}=\llave{0}$.

If $p\in\Cor\parent{\delta_{k}}\cap\Ker\parent{\tilde{\Opl}_{n+k}}$ then $\tilde{\Opl}_{n+k}\parent{p}=0$, that is, $p$ is a first integral of the vector field $\G_{n}:=\F_{n}-\fracp{K_{n}}{k}\FZERO$ hence $p$ is vanished in all the invariant curves of $\G_{n}$ which are the same as those of $\F_{n}$ therefore $p\in\Range\parent{\delta_{k}}\cap\Cor\parent{\delta_{k}}=\llave{0}$.

\item[ii)] Secondly we prove that $\dim\parent{\tilde{\Opl}_{n+k}\parent{\Cor\parent{\delta_{k}}}}=
    \dim\parent{\qh_{n+k}^\t}-\dim\parent{\Range\parent{\delta_{n+k}}}$.

Taking into account that $\delta_{n+k}$ is injective we have
\[
\dim\parent{\qh_{n+k}^\t}-\dim\parent{\Range\parent{\delta_{n+k}}}=
\dim\parent{\qh_{n+k}^\t}-\dim\parent{\qh_{n+k-s}^\t}.
\]
As $\delta_{k}$ is injective and $\Cor\parent{\delta_{k}}\cap\Ker\parent{\tilde{\Opl}_{n+k}}=\llave{0}$, the operator $\tilde{\Opl}_{n+k}$ is also injective over $\Cor\parent{\delta_k}$ and therefore
\bean \dim\parent{\tilde{\Opl}_{n+k}\parent{\Cor\parent{\delta_{k}}}}&=&\dim\parent{\Cor\parent{\delta_{k}}}=
\dim\parent{\qh_{k}^\t}-\dim\parent{\Range\parent{\delta_{k}}}\\
&=&\dim\parent{\qh_{k}^\t}-\dim\parent{\qh_{k-s}^\t}
\eean

It is sufficient to prove that $\dim\parent{\qh_{n+k}^\t}-\dim\parent{\qh_{n+k-s}^\t}=
\dim\parent{\qh_{k}^\t}-\dim\parent{\qh_{k-s}^\t}$.
From Lemma \ref{lemm} the irreducible invariant curves of $\F_{n}$ are the irreducible factors of the Hamiltonian which as we know are the factors $x$, $y$, $y^{t_1}-\lambda x^{t_2}$, $\lambda\in\mathbb{R}\setminus\llave{0}$ or $(y^{t_1}-ax^{t_2})^2+b^2x^{2t_2}$, $a,b\in\mathbb{R}$, $b\neq 0$ and therefore $s=t_1$, $s=t_2$, $s=t_1t_2$ or $s=2t_1t_2$ respectively. Lemma 2.3 in \cite{AlgabaNonlinearity09} proves the case for $s=t_1t_2$ and for $s=2t_1t_2$. The case $s<t_1t_2$ is excluded because we do not consider the case when $x=0$ or $y=0$ be invariant curves.

\medskip

Up to now we have proved that $\qh_{n+k}^\t=\tilde{\Opl}_{n+k}\parent{\Cor\parent{\delta_k}}
\bigoplus\Range\parent{\delta_{n+k}}$ hence given $q\in\qh_{n+k}^\t$ there exists a unique $p_1\in\tilde{\Opl}_{n+k}\parent{\Cor\parent{\delta_k}}$ and a unique $p_2\in\Range\parent{\delta_{n+k}}$ such that $q=p_1+p_2$. On the other hand, since $\delta_{n+k}$ is injective and $\Cor\parent{\delta_{k}}\cap\Ker(\tilde{\Opl}_{n+k})=\llave{0}$, the operator $\tilde{\Opl}_{n+k}$ is also injective over $\Cor(\delta_k)$, and we can
assert that there exist a unique $p^{(1)}\in\Cor\parent{\delta_k}$ and $p^{(2)}\in\qh_{n+k-s}$ such that $p_1=\tilde{\Opl}_{n+k}\parent{p^{(1)}}$ and a unique $p_2=p^{(2)}f$ and this completes the proof.
\end{itemize}\end{proof}

The following result gives us information about the formal invariant curves of a nilpotent vector field whose first quasi-homogeneous component is non-conservative.

\begin{theorem}\label{teocuinvnilp} Consider $n\in\mathbb{N}$ and $\F=\F_n+\cdots=\sum_{j\geq n}\F_j$, $\F_j\in\QH_j^{(1,n+1)}$ and $\F_n=(y+dx^{n+1},(n+1)x^{2n+1}+(n+1)dx^ny)^T$. If $|d|\neq 1+\fracp{2(n+1)}{k-n-1}$ for all $k\in\mathbb{N}$, $k>n+1$, then there exist only two formal invariant curves which have the following form $C:=C_{n+1}+\sum_{j> n+1} C_j$ where $C_{n+1}=y-x^{n+1}$ or $C_{n+1}=y+x^{n+1}$ and $C_j\in\Cor\parent{\delta_{j}}$, being $\delta_j$ the linear operator  defined in Lemma \ref{leqrangprang} for $f=C_{n+1}$.
\end{theorem}
\begin{proof} Note that $y-x^{n+1}$ and $y+x^{n+1}$ are the unique irreducible invariant curves of $\F_n$.
We must prove the existence of a cofactor of the form $K=\sum_{j\geq n}K_j$, $K_j\in\qh_j^\t$ such that
\bea\label{ecuCSexistcurvainv} \nabla C\cdot\F-CK&=&0.\eea

We will check by induction method that the equation (\ref{ecuCSexistcurvainv}) is satisfied degree by degree, for which we will choose appropriately $C_j$ and $K_{j-1}$ in each case.

Equation (\ref{ecuCSexistcurvainv}) for degree $2n+1$ is satisfied since we consider $K_n$ the cofactor of $C_{n+1}$, that is,
\bean \nabla C_{n+1}\cdot\F_n-C_{n+1}K_n&=&0.\eean

We assume that equation (\ref{ecuCSexistcurvainv}) is satisfied for degree $n+j-1$ and we want to prove that it is also satisfied for degree $n+j$ with $j>n+1$.
\bean
\parent{\nabla C\cdot\F-C K}_{n+j}&=&\nabla C_j\F_n+\sum_{i=1}^{j-n-1}\nabla C_{j-i}\cdot\F_{n+i}\\
&&-C_j K_n-C_{n+1}K_{j-1}-\sum_{i=1}^{j-n-2}C_{j-i}K_{n+i}\\
&=&\nabla C_j\parent{\F_n-\fracp{1}{j}K_n\FZERO}- C_{n+1} K_{j-1}+R_{n+j}.
\eean
where $R_{n+j}=\sum_{i=1}^{j-n-1}\nabla C_{j-i}\cdot\F_{n+i}-\sum_{i=1}^{j-n-2}C_{j-i}K_{n+i}$ is a quasi-homogeneous polynomial of degree $n+j$ determined by the previous values of $C_i$ and $K_{n+i}$.

Taking into account that $|d|\neq 1+\fracp{2(n+1)}{k-n-1}$ for all $k>n+1$, by applying Lemma \ref{leqrangprang} for $k=j$, we obtain $\tilde{\Opl}_{n+j}\parent{\Cor\parent{\delta_{j}}}\cap\Range\parent{\delta_{n+j}}=\llave{0}$ and by Lemma \ref{leqhsumadirecta} there exits a unique $R^{(1)}\in\Cor\parent{\delta_{j}}$ and a unique $R^{(2)}\in\qh_{j-1}^\t$ such that
\bean R_{n+j}&=&\tilde{\Opl}_{n+j}\parent{R^{(1)}}+ C_{n+1} R^{(2)}.\eean
Hence we have
\bean \parent{\nabla C\cdot\F-C K}_{n+j}&=&\tilde{\Opl}_{n+j}\parent{C_j}- C_{n+1} K_{j-1}+\tilde{\Opl}_{n+j}\parent{R_j^{(1)}}+ C_{n+1} R_{j-1}^{(2)}\\
&=&\tilde{\Opl}_{n+j}\parent{C_j+R_j^{(1)}}+ C_{n+1} \parent{R_{j-1}^{(2)}-K_{j-1}}.\eean
Therefore in order to satisfy equation (\ref{ecuCSexistcurvainv}) for degree $n+j$, it is sufficient to chose $K_{j-1}=R_{j-1}^{(2)}$ and $C_j=-R_j^{(1)}$.

As $\Ker\parent{\tilde{\Opl}_{r+j}}\subset\Range\parent{\delta_{j}}$ we have the uniqueness.
\end{proof}

\subsection{Necessary conditions of integrability}

For our study we will need the following results
\begin{proposition}\cite[Corollary  3.1]{A3}\label{cni}
If $\F_r:=\X_{\hx}+\mu\FZERO\in\QH_r^\t$ is irreducible and $\hx$ has multiple factors in its
decomposition over $\mathbb{C}[x,y]$, $\hx\neq$cte. and $\mu \not\equiv
0$, then $\F_r$ is non-analytically integrable.
\end{proposition}

\begin{theorem}\cite[Theorem 3.2]{A3}\label{ci}
Assume that $\F_r:=\X_{\hx}+\mu\FZERO\in\QH_r^\t$ is irreducible and $\hx$
has two or more than two factors, all of then simples in its
decomposition over $\mathbb{C}[x,y]$. System $\dot{\x}={\bf F}_r(\x)$
has a first integral  if and only if either $\mu\equiv 0$ or there
exists $n_x,n_y$, $n_i$, $i=1,\cdots, m$ non-negative numbers, not
all zeros, such that
 \bea\left\{\begin{array}{lr}Res[{\eta^{\rm
 hom}(X,1)},0]=\frac{(n_x+1)(r+|\t|)-M_0}{t_2M_0},
 &\mbox{ if } \delta_x=1,\\
 Res[{\eta^{\rm hom}(1,Y)}, 0]=-\frac{(n_y+1)(r+|\t|)-M_0}{t_1M_0},&\mbox{ if } \delta_y=1,\\
 Res[{\eta^{\rm hom}(1,Y)}, \lambda_i]=-\frac{(n_{i}+1)(r+|\t|)-M_0}{M_0},&
 i=1,\cdots,m,
 \end{array}\right. \eea
with $\eta^{\rm hom}(X,Y)=\frac{\mu^{\rm hom}(X,Y)}{X^{\delta_x}Y^{\delta_y}\hx^{\rm hom}(X,Y)}$,
$M_0=t_1(n_x+1)\delta_x+t_2(n_y+1)\delta_y+t_1t_2\sum_{j=1}^m(n_j+1)$ where $f^{\rm hom}$ for $f\in\qh_s^\t$, is the unique homogeneous polynomial verifying $f(x,y)=x^ny^mf^{\rm hom}(x^{t_2},y^{t_1})$ with $n,m\in\mathbb{N}\cup\llave{0}$.

Moreover, in this case, a first integral of degree $M_0$ is
\[
I=x^{(n_x+1)\delta_x}y^{(n_y+1)\delta_y}\prod_{i=1}^m
(y^{t_1}-\lambda_i x^{t_2})^{n_i+1}.
\]
\end{theorem}

Now we study the integrability problem for nilpotent vector fields.

\begin{proposition}[Necessary condition of integrability]\label{pronilint}

Let $\F=\sum_{j\geq r}\F_j$, $\F_j\in\QH_j^\t$ be a nilpotent vector field such that the origin of $\dot{\x}=\F_r(\x)$ is isolated. If $\F$ is analytically integrable,  then there exists $\Phi_0\in\QH_0^\t$, $\det\parent{D\Phi_0(\zero)}\neq 0$ such that $\G:=\parent{\Phi_0}_*\F=\G_r+\cdots$, where $\G_r\in\QH_r^\t$, $\cdots$ are quasi-homogeneous terms of type $\t$ and degree greater than $r$ and $\G_r$ is one of the following vector fields.
\begin{description}
\item[i)]   $\G_r=\vedo{y}{x^{2n}}\in\QH_{r}^{\t}$, with  $\t=(2,2n+1)$, $r=2n-1$ and $I_{2(2n+1)}=2x^{2n+1}-(2n+1)y^2\in\qh_{2(2n+1)}^\t$ is a first integral of $\G_r$.
\item[ii)]  $\G_r=\vedo{y}{\sigma(n+1)x^{2n+1}}\in\QH_{r}^{\t}$, with $\t=(1,n+1)$, $r=n$, $\sigma=\pm 1$ and $I_{2(n+1)}=  -\sigma x^{2n+2}+y^2\in\qh_{2(n+1)}^\t$ is a first integral of $\G_r$.
\item[iii)] $\G_r=\vedo{y}{(n+1)x^{2n+1}}+dx^{n}\vedo{x}{(n+1)y}\in\QH_{r}^{\t}$, with $\t=(1,n+1)$, $r=n$, $d=\fracp{m_1-m_2}{m_1+m_2}\neq 0$, $m_1,m_2\in\mathbb{N}$, $m_1,m_2$, coprimes, and $I_{M}=(y-x^{n+1})^{m_1}(y+x^{n+1})^{m_2}\in\qh_{M}^\t$, with $M=(n+1)(m_1+m_2)$ is a primitive first integral of $\G_r$.
\end{description}
\end{proposition}
\begin{proof} If the origin is an isolated singularity of $\dot{\x}=\F_r(\x)$, by applying Proposition \ref{propreFN}, there exists $\Phi_0\in\QH_0^\t$, $\det\parent{D\Phi_0(\zero)}\neq 0$ such that $\G:=\parent{\Phi_0}_*\F=\G_r+\cdots$  and $\G_r$ is of type {\bf B1), {\bf B2)}, {\bf B3)} or {\bf B4)}}. Moreover if $\F$ is integrable $\G$ is integrable and by Lemma \ref{leintFintFr} $\G_r$ is integrable.
\begin{itemize}
\item If $\G_r$ is of the form {\bf B1)} we have that $\G_r=(y,x^{2n})^T\in\QH_r^\t$ where $\t=(2,2n+1)$ and $r=2n$,  which is a quasi-homogeneous Hamiltonian vector field therefore integrable with a first integral of the form $I_{2(2n+1)}=2x^{2n+1}-(2n+1)y^2\in\qh_{2(2n+1)}^\t$. This corresponds to case {\bf i)}.

\item If $\G_r$ is of the form {\bf B2)}, we have that $\G_r=(y,-(n+1)x^{2n+1})^T+dx^n(x,(n+1)y)^T\in\QH_r^\t$ is irreducible, where $\t=(1,n+1)$ and $r=n$. In this case the Hamiltonian function and the divergence of $\G_r$ are $\hx=-y^2/2-x^{2n+2}/2=-1/2(y-ix^{n+1})(y+ix^{n+1})$, and $\mu=dx^n$. Applying Theorem \ref{ci} must be fulfilled $d=0$ or $d\neq 0$ with
\bean
\Res{\eta(1,y)}{i}=\fracp{d}{-i}=1-\fracp{(n_1+1)2(n+1)}{(n+1)(n_1+n_2+2)},\\
\Res{\eta(1,y)}{-i}=\fracp{d}{i}=1-\fracp{(n_2+1)2(n+1)}{(n+1)(n_1+n_2+2)}.
\eean
Such conditions are not satisfied so it should be $d=0$ and consequently $\G_r$ is a quasi-homogeneous Hamiltonian vector field with a first integral of the form $I_{2(n+1)}=x^{2n+2}+y^2$. This case corresponds to case {\bf ii)} for $\sigma=-1$.

\item If $\G_r$ is of the form {\bf B3)}, we have that $\G_r=(y,0)^T+dx^n(x,(n+1)y)^T\in\QH_r^\t$ with $d\neq 0$ is irreducible, where $\t=(1,n+1)$ and $r=n$. In this case the Hamiltonian function is $\hx=-\fracp{1}{2}y^2$ which has multiple factors . Applying Proposition \ref{cni} we have that $\G_r$ is not integrable and therefore neither it is $\F$.

\item If $\G_r$ is of the form {\bf B4)}, we have that $\G_r=(y,(n+1)x^{2n+1})^T+dx^n(x,(n+1)y)^T$ with $d\neq 1$ is irreducible. In this case the Hamiltonian functions and the divergence of $\G_r$ are $\hx=-\fracp{1}{2}(y-x^{n+1})(y+x^{n+1})$ and $\mu=dx^n$. Applying Theorem \ref{ci} $\G_r$ is integrable if, and only if, $d=0$ or $d\neq 0$ with
\bean \Res{\eta(1,y)}{1}=\fracp{d}{-\frac{1}{2}(2)}=1-\fracp{(n_1+1)2(n+1)}{(n+1)(n_1+n_2+2)}=\fracp{n_2-n_1}{n_1+n_2+2},\\
\Res{\eta(1,y)}{-1}=\fracp{d}{-\frac{1}{2}(-2)}=1-\fracp{(n_2+1)2(n+1)}{(n+1)(n_1+n_2+2)}=-\fracp{n_1-n_2}{n_1+n_2+2}\eean
If $d=0$ we have a first integral of the form $I_{2(n+1)}=-x^{2n+2}+y^2$. This case corresponds to case {\bf ii)} for $\sigma=1$. If $d\neq 0$ we consider $m_1:=n_1+1$, $m_2:=n_2+1$, then $d=\fracp{m_1-m_2}{m_1+m_2}\neq 0$ with $m_1,m_2\in\mathbb{N}$, coprimes and a first integral is $I_M=(y-x^{n+1})^{m_1}(y+x^{n+1})^{m_2}\in\qh_M^\t$ with $M=(n+1)(m_1+m_2)$. $I_M$ is a primitive first integral since $m_1,m_2$ are coprimes. This corresponds to case {\bf iii)}.
\end{itemize}
\end{proof}
\begin{remark}
The study of integrability for the cases {\bf i)} and {\bf ii)} was already solved in \cite{AlgabaNonlinearity09}. Therefore the only case to study is the case {\bf iii)}. That is, we focus in the study of the analytic integrability of vectors fields of the form
\bea\label{SistObjestu1}
\F&=&\sum_{j\geq n} \F_j, \quad \F_j\in\QH_j^\t, \ n\in\mathbb{N}, \ \ \t=(1,n+1),
\eea
where
\bea
\nonumber \F_n&=&\vedo{y+dx^{n+1}}{(n+1)x^{2n+1}+(n+1)dx^ny}, \, d=\fracp{m_1-m_2}{m_1+m_2}, \, m_1,m_2\in\mathbb{N}, \mbox{ coprimes}.
\eea
\end{remark}

The following theorem provides a necessary condition on the analytic integrability of system (\ref{SistObjestu1}).

\begin{theorem}\label{teointprimeranilp}
If the vector field (\ref{SistObjestu1}) is analytically integrable, then a first integral of $\F$ is $I=I_M+\cdots$ with $I_M=(y-x^{n+1})^{m_1}(y+x^{n+1})^{m_2}\in\qh_M^\t$, $M=(n+1)(m_1+m_2)$.
\end{theorem}
\begin{proof} If $\F$ is integrable, $\F_n$ has a quasi-homogeneous first integral and by Proposition \ref{pronilint} there exist $m_1,m_2\in\mathbb{N}$ coprimes, such that $d=\fracp{m_1-m_2}{m_1+m_2}$ and a first integral of $\F_n$ is  $I_M=(y-x^{n+1})^{m_1}(y+x^{n+1})^{m_2}\in\qh_M^\t$ with $M=(n+1)(m_1+m_2)$.

On the other hand $d=\fracp{m_1-m_2}{m_1+m_2}$ therefore $|d|=\fracp{|m_1-m_2|}{m_1+m_2}<1$ and consequently $|d|\neq 1+\fracp{2(n+1)}{k-n-1}$ for all $k>n+1$.

Applying Theorem \ref{teocuinvnilp} we can assert that there exist only two irreducibles formal invariant curves of $\F$ passing through the origin and having the form $C_1=(y-x^{n+1})+\cdots$ and $C_2=(y+x^{n+1})+\cdots$. Therefore a first integral of $\F$ must be of the form $\tilde{I}=C_1^{s_1}C_2^{s_2}u$ where $u=1+\cdots$ is a unity. Considering the lower quasi-homogeneous degree results that $(y-x^{n+1})^{s_1}(y+x^{n+1})^{s_2}$ is a first integral of $\F_n$ and since $I_M$ is also a first integral of $\F_n$ it must exist $l\in\mathbb{N}$ such that $s_1=lm_1$, $s_2=lm_2$. Hence
\bean\tilde{I}&=&C_1^{lm_1}C_2^{lm_2}u=\parent{C_1^{m_1}C_2^{m_2}u^{1/l}}^l\eean
Therefore if we denote by $I=C_1^{m_1}C_2^{m_2}u^{1/l}=I_M+\cdots$ this function is also a first integral of $\F$.
\end{proof}

\section{Normal Form}

The analysis of normal forms for planar vector fields and related
questions as center problem or integrability problem has been considered
in \cite{AlgabaNonlinearity09}. In this section we determine an orbital normal form for a general vector field $\F=\F_n+\cdots$, where the lowest degree quasi-homogeneous of $\F$, $\F_n=\X_{\hx}+\mu\FZERO\in\QH_n^\t$,
 with $\hx\in\qh_{n+|\t|}^\t$ and $\mu\in\qh_n^\t$, conservative-dissipative decomposition of $\F_n$), that is, through the paper $\hx$ will represent the conservative part of $\F_n$ and $\mu$ will be the dissipative part of $\F_n$, see expression (\ref{componentes}).

A conjugated normal form of a nilpotent singularity in the generalized saddle case is studied in \cite{BaiderSanders1992} and an orbital normal form in \cite{Strozyna_Zoladek_2003}. In this work we show, in Theorem \ref{NFhfsimples}, another orbital normal form in the generalized saddle more convenient for the analysis of the integrability problem.

 The following subspaces of $\QH_k^\t$ will be useful in the study of
the homological  operator under orbital equivalence that we will see later.
 \bean
 \calc_k^\t&:=&\llave{\X_{g_{k+|\t|}}: g_{k+|\t|}\in\Delta_{k+|\t|} \mbox{, a complementary subspace to} \ \hx\qh_{k-n}^\t},\\
\cald_k^\t&:=&\llave{\eta_k\FZERO : \eta_k\in\qh_{k}^\t},\\
\calf_k^\t&:=&\llave{\lambda_{k-n}\F_{n} : \lambda_{k-n}\in\qh_{k-n}^\t}.
 \eean

Beyond the splitting, the following linear operator also plays an important
role in this study (see also  \cite{AlgabaNonlinearity09}).
\bea\label{DefOpl}\begin{array}{rcl}\Opl_{k}&:&\qh_{k-n}^{\t}
\longrightarrow \qh_{k}^{\t}\\
 && \mu_{k-n} \longrightarrow \nabla\mu_{k-n}\cdot\F_{n},\end{array}
\eea

In order to calculate a normal form of the vector field we need a new decomposition of the quasi-homogeneous vector fields
\begin{lemma} If $\hx\neq 0$, then $\QH_k^\t=\calc_k^\t\bigoplus\cald_k^\t\bigoplus\calf_k^\t$, for all $k\in\N$. Moreover if $\P_k\in\QH_k^\t$, there exist $g\in\Delta_{k+|\t|}$, $\eta\in\qh_k^\t$ and $\lambda\in\qh_{k-n}^\t$ such that $$\P_k=\X_{g}+\eta\FZERO+\lambda\F_n,$$ where
\bean g=\fracp{\Proy_{\Delta_{k+|\t|}}\FZERO\wedge\P_k}{k+|\t|},\ \
\lambda=\fracp{\Proy_{\hx\qh_{k-n}^\t}\FZERO\wedge\P_k}{(n+|\t|)\hx},\ \ \eta=\fracp{\diverg\parent{\P_k}-\nabla\lambda\cdot\F_n-\lambda\diverg\parent{\F_n}}{n+|\t|}.
\eean
\end{lemma}

\begin{proof}
It is obvious that
$\calc_k^\t+\cald_k^\t+\calf_k^\t\subseteq\QH_k^\t$.

Moreover:
 \begin{itemize}
 \item
$\calc_k^\t\cap\cald_k^{\t}=\llave{\zero}$ trivially (see
(\ref{componentes})).
 \item
$\cald_k^\t\cap\calf_k^\t=\llave{\zero}$, because otherwise, there
is some $\P_k\in\cald_k^\t\cap\calf_k^\t\setminus\llave{\zero}$. Then, one can find
$\lambda_{k-n}\in\qh_{k-n}^\t\setminus\llave{0}$ such that
$\P_k=\lambda_{k-n}\F_{n}$, and there
exists $\mu_k\in\qh_k^\t$ such that $\P_k=\mu_k\FZERO$. Hence,
$0=\FZERO\wedge\parent{\mu_k\FZERO}=\FZERO\wedge\P_k=\FZERO\wedge\parent{\lambda_{k-n}\F_n}=(n+|\t|)\lambda_{k-n}\hx$, which is a contradiction.
 \item
$\calc_k^\t\cap\calf_k^\t=\llave{\zero}$, because otherwise, there
is some $\P_k\in\calc_k^\t\cap\calf_k^\t\setminus\llave{\zero}$.
Then, $\P_k=\X_{\gx}=\lambda_{k-n}\F_{n}$ where
$\lambda_{k-n}\in\qh_{k-n}^\t\setminus\llave{0}$ and
$\gx\in\Delta_{k+|\t|}\setminus\llave{0}$. Therefore $\P_k = \lambda_{k-n} \F_{n} = \lambda_{k-n} \X_{\hx} + \lambda_{k-n} \mu
\FZERO$ and applying (\ref{componentes}) we obtain $\lambda_{k-n}\X_{\hx}=\fracp{n+|\t|}{k+|\t|} \X_{\lambda_{k-n}\hx} +\fracp{1}{k+|\t|}
\parent{\nabla\lambda_{k-n}\cdot\X_{\hx}}\FZERO$ and then
\bea\label{ecumukr}\P_k&=& \X_{\fracp{n+|\t|}{k+|\t|}\lambda_{k-n}\hx} +
\parent{\fracp{1}{k+|\t|} \parent{\nabla\lambda_{k-n}\cdot\X_{\hx}}
+ \lambda_{k-n} \mu}\FZERO,
\eea
Using (\ref{componentes}), we get
$\gx=\fracp{n+|\t|}{k+|\t|}\lambda_{k-n}\hx$. Hence
$\gx\in\Delta_{k+|\t|}\cap\hx\qh_{k-n}^\t=\llave{0}$,
which is a contradiction.
\end{itemize}

Consequently, it {remains} to show that
$\QH_k^\t\subseteq\calc_k^\t+\cald_k^\t+\calf_k^\t$. Let us consider
$\P_k\in\QH_k^\t$. From (\ref{componentes}), we can write
$\P_k=\X_{\hgx}+\mu_k\FZERO$, for some $\hgx\in\qh_{k+|\t|}^\t$,
$\mu_k\in\qh_k^\t$. As
$\qh_{k+|\t|}^\t=\Delta_{k+|\t|}\bigoplus\hx\qh_{k-n}^\t$,
we can also write $\hgx=\gx+\tgx\hx$ for some
$\gx\in\Delta_{k+|\t|}$, $\tgx\in\qh_{k-n}^\t$. Then,
$\P_k=\X_{\gx}+\X_{\tgx\hx}+\mu_k\FZERO$. Using (\ref{ecumukr}) for
$\tgx \F_n$, we obtain $\tgx\F_{n}=\X_{\fracp{n+|\t|}{k+|\t|}\tgx\hx}
+\parent{\fracp{1}{k+|\t|}\parent{\nabla\tgx\cdot\X_{\hx}}+\tgx\mu}\FZERO$.
So, $\X_{\tgx\hx} = \fracp{k+|\t|}{n+|\t|}\tgx\F_{n} -
\parent{\fracp{1}{n+|\t|}\parent{\nabla\tgx\cdot\X_{\hx}}+
\fracp{k+|\t|}{n+|\t|}\tgx\mu}\FZERO$, and then
 \bean
\P_k&=&\X_{\gx}+\parent{\mu_k-\fracp{1}{n+|\t|}\parent{\nabla\tgx\cdot\X_{\hx}}-
\fracp{k+|\t|}{n+|\t|}\tgx\mu}\FZERO
+\fracp{k+|\t|}{n+|\t|}\tgx\F_{n}.\eean
Only remains to find the expressions of $g$, $\eta$ and $\lambda$.
$$\D_0 \wedge \P_k = \D_0 \wedge (\X_{g} + \eta\D_0 + \lambda\F_n) = (k+|\t|)g + (n+|\t|)\lambda\hx$$
Therefore $g =\frac{Proy_{\Delta_{k+|\t|}}(\D_0 \wedge
\P_k)}{k+|\t|}$ and $\lambda =
\frac{Proy_{\hx\qh_{k-n}^{\t}}(\D_0 \wedge \P_k)}{(n+|\t|)\hx}.$
On the other hand  $\diverg(\P_k)=(k+|\t|)\eta+
\nabla\lambda\cdot\F_n+\lambda\diverg(\F_n)$, that is,
$\eta=\frac{\diverg(\P_k)-\nabla
\lambda\F_n-\lambda\diverg(\F_n)}{k+|\t|}$.
\end{proof}

The above result allows to define the corresponding projectors
$\Pi^{\co}$, $\Pi^{\di}$  and $\Pi^{\fp}$. Also, we can identify
$\QH_k^\t = \calc_k^\t \bigoplus \cald_k^\t \bigoplus \calf_k^\t
\equiv \calc_k^\t \times \cald_k^\t \times \calf_k^\t$. We will
denote $\Pi^{\di}\parent{\P_k}=\P_k^{\di}:=
\Proy_{\cald_k^{\t}}\parent{\P_k}$,
$\Pi^{\co}\parent{\P_k}=\P_k^{\co}:=
\Proy_{\calc_k^{\t}}\parent{\P_k}$, and
$\Pi^{\fp}\parent{\P_k}=\P^{\fp}_k:=
\Proy_{\calf_k^{\t}}\parent{\P_k}$.

\smallskip

With this notation the homological operator under equivalence, defined as
$$\begin{array}{rcl}\Lequiv_{n+k}:\QH_k^\t\times\Cor\parent{\Opl_k}&\rightarrow&\QH_{n+k}^\t\\
\Lequiv_{n+k}\parent{\P_k,\nu_k}&=&-[\F_n,\P_k]-\nu_k\F_n\end{array}$$
can be written as
$$ \Lequiv_{n+k}:\calc_k^\t\times
\cald_{k}^\t \times \calf_k^\t\times \Cor\parent{\Opl_k}
\longrightarrow
\calc_{n+k}^\t\times\cald_{n+k}^\t\times\calf_{n+k}^\t,
$$
such as
 $$
\Lequiv_{n+k}\parent{\P_{k}^{\co},\P_{k}^{\di},\P_k^{\fp},\nu_k}=
-\parent{\bracket{\F_{n},\P_{k}^{\co}}^{\co},
\bracket{\F_{n},\P_{k}^{\co}+\P_{k}^{\di}}^{\di},\bracket{\F_{n},\P_{k}^{\co}+\P_{k}^{\di}+\P_{k}^{\fp}}^{\fp}+ \nu_k\F_{n}},
 $$
where $\P_k^{\co}=\X_{\gx}$ with $\gx\in\Delta_{k+|\t|}$,
$\P_k^{\di}=\mu_k\FZERO$ with $\mu_k\in\qh_k^\t$,
$\P_k^{\fp}=\lambda_{k-n}\F_{n}$ with $\lambda_{k-n}\in\qh_{k-n}^\t$
and $\nu_k\in\Cor\parent{\Opl_{k}}$.
\begin{lemma}\label{lecorhliecdf}
Let us consider $\lambda_{k-n} \F_{n}\in\calf_k^\t$,
$\eta_k \FZERO\in\cald_k^\t$, $\X_{\gx}\in\calc_k^\t$. Then:
\begin{description}
\item[(a)]
$\bracket{\F_{n},\lambda_{k-n} \F_{n}} = -
\Opl_{k}\parent{\lambda_{k-n}} \F_{n}\in\calf_{n+k}^\t$.
\item[(b)]
$\bracket{\F_{n},\eta_k \FZERO} = -
\Opl_{n+k}\parent{\eta_k}\FZERO+n \eta_k \F_{n}
\in\cald_{n+k}^\t\bigoplus\calf_{n+k}^\t$.
\item[(c)]
$\bracket{\F_{n},\X_{\gx}}^{\co} = -
\X_{\parent{\Opl_{n+k+|\t|}^{\co}\parent{\gx}}}$ where
$\Opl_{n+k+|\t|}^{\co} : \Delta_{k+|\t|} \longrightarrow
\Delta_{n+k+|\t|}$ is defined by
\bea\label{DefOplc}
\Opl_{n+k+|\t|}^{\co}\parent{\gx} = \Proy_{\Delta_{n+k+|\t|}}
\parent{ \nabla\gx\cdot\parent{\F_{n}-\fracp{n+|\t|}{n+k+|\t|}\mu\FZERO}}.
\eea
\end{description}
\end{lemma}

\begin{proof}
Items {\bf (a)}, {\bf (b)}, are
consequences of following properties $[\mu\F,\G]=\parent{\nabla \mu\cdot\G}\F+\mu[\F,\G]$, $[\F_k,\FZERO]=k\F_k$,
respectively. From the properties of the Lie bracket and the previous properties we deduce
 \bean
\bracket{\F_{n},\X_{\gx}} & = &
\bracket{\X_{\hx},\X_{\gx}}+\bracket{\mu\FZERO,\X_{\gx}}=\X_{-\nabla\gx\cdot\X_{\hx}}
+
\parent{\nabla\mu\cdot\X_{\gx}}\FZERO+\mu\bracket{\FZERO,\X_{\gx}}
\\
&=&\X_{-\nabla\gx\cdot\X_{\hx}} +
\parent{\nabla\mu\cdot\X_{\gx}}\FZERO-k\mu\X_{\gx}\\
&=&\X_{-\nabla\gx\cdot\X_{\hx}} +
\parent{\nabla\mu\cdot\X_{\gx}}\FZERO-\fracp{k(k+|\t|)}{n+k+|\t|}\X_{\mu\gx}
-\fracp{k}{n+k+|\t|}\parent{\nabla\mu\cdot\X_{\gx}}\FZERO
\\
&=&-\X_{\nabla\gx\cdot\X_{\hx}+\fracp{k(k+|\t|)}{n+k+|\t|}\mu\gx} +
\fracp{n+|\t|}{n+k+|\t|}\parent{\nabla\mu\cdot\X_{\gx}}\FZERO\\
&=&-\X_{\nabla\gx\cdot\parent{\F_{n}-\fracp{n+|\t|}{n+k+|\t|}\mu\FZERO}} +
\fracp{n+|\t|}{n+k+|\t|}\parent{\nabla\mu\cdot\X_{\gx}}\FZERO,
 \eean
which implies item {\bf (c)}.
\end{proof}

Next result provides us an orbital normal form of the vector field $\F$ with first quasi-homogeneous component non-conservative.
\begin{theorem}\label{teoFNlequivFn} Let $\F=\sum_{j\geq n}\F_j$, $\F_j\in\QH_j^\t$. If $\Ker\parent{\Opl_{n+k+|\t|}^{\co}}=\llave{0}$ for all $k\in\mathbb{N}$ then $\F$ is orbitally equivalent to $$\G=\F_n+\sum_{j>n}\G_j, \ \mbox{with} \ \G_j=\X_{g_{j+|\t|}}+\eta_j\FZERO\in\QH_j^\t,$$
 where $g_{j+|\t|}\in\Cor\parent{\Opl^{c}_{n+j+|\t|}}$ and $\eta_j\in\Cor\parent{\Opl_{j}}$, the operators $\Opl^{c}_{n+j+|\t|}$ and $\Opl_j$ are defined in (\ref{DefOplc}) and (\ref{DefOpl}), respectively.
\end{theorem}

\begin{proof}
It is sufficient to show that
\bean
\Cor\parent{\Lequiv_{n+k}}&=&
\X_{\Cor\parent{\Opl_{n+k+|\t|}^{\co}}}
\bigoplus\Cor\parent{\Opl_{n+k}}\FZERO,
 \eean
is a complementary subspace to the range of $\Lequiv_{n+k}$.

From Lemma \ref{lecorhliecdf}, we get
$
\Lequiv_{n+k}\parent{\P_{k}^{\co},\P_{k}^{\di},\P_k^{\fp},\nu_k}=$
\bean
\parent{\X_{\parent{\Opl_{n+k+|\t|}^{\co}\parent{\gx}}},
-\bracket{\F_{n},\X_{\gx}}^{\di} +
\Opl_{n+k}\parent{\mu_k}\FZERO,-\bracket{\F_{n},\X_{\gx}}^{\fp}-\mu_k\F_{n}
+\Opl_{k}\parent{\lambda_{k-n}}\F_{n}- \nu_k\F_{n}}.
 \eean

By using $\calf_{n+k}^\t = \Range\parent{\Opl_{k}}\F_{n} \bigoplus
\Cor\parent{\Opl_{k}}\F_{n}$, and $\qh_{k}^\t =
\Range\parent{\Opl_{k}} \bigoplus \Cor\parent{\Opl_{k}}$, we
obtain the following scheme for $\Lequiv_{n+k}$:
\renewcommand{\arraystretch}{1.5}
 $$\begin{array}{|c|c|c|c|c|c|}
 \hline
\cellcolor[gray]{.9}\calc_{k}^\t & \cellcolor[gray]{.9}\cald_{k}^\t & \cellcolor[gray]{.9}\calf_k^\t &
\cellcolor[gray]{.9}\Range\parent{\Opl_{k}}& \cellcolor[gray]{.9}\Cor\parent{\Opl_{k}} &
 \  \\
\hline
 \X_{\parent{\Opl_{n+k+|\t|}^{\co}\parent{\gx}}} &
 0 &
 0 &
 0 &
 0 &
 \cellcolor[gray]{.9}\calc_{n+k}^\t \\
 \hline
 \bullet &
 \Opl_{n+k}\parent{\mu_k}\FZERO &
 0 &
 0 &
 0 &
 \cellcolor[gray]{.9}\cald_{n+k}^\t \\
 \hline
 \multirow{2}{*}{$\bullet$} &
 \multirow{2}{*}{$\bullet$} &
  { \Opl_{k}\parent{\lambda_{k-n}}\F_{n}} &
  {-\Opl_{k}\parent{\rho_{k-n}}\F_{n}} &
  {0} &
  \cellcolor[gray]{.9}{\Range\parent{\Opl_{k}}\F_{n}} \\
  \cline{3-6}
   &
   &
 {0} &
  {0} &
  {-\nu_{k}\F_{n}} &
  \cellcolor[gray]{.9}{\Cor\parent{\Opl_{k}}\F_{n}}  \\
 \hline
\end{array}$$
From hypothesis $\Ker\parent{\Opl_{n+k+|\t|}^{\co}}=\llave{0}$, we can deduce that the upper left block diagonal of
the above matrix has maximum range. The result follows from the structure of the
above matrix.
\end{proof}

\begin{remark} The case {\bf A)} of Proposition \ref{propreFN}, i.e., the vector field $\F=(y,bx^ny)^T+\cdots$ does not satisfy the hypothesis of Theorem \ref{teoFNlequivFn}. More specifically , this vector field does not verify the condition $\Ker\parent{\Opl_{n+k+|\t|}^{\co}}=\llave{0}$.
\end{remark}

\subsection{Normal form for a perturbation of a quasi-homogeneous non-hamiltonian nilpotent vector field}

In order to determine an orbital normal form for a nilpotent vector field whose first quasi-homogeneous component is non-hamiltonian we need the following result.

\begin{lemma}\label{lebasepkt}
Consider $\F_n=\X_{\hx}+\mu\FZERO\in\QH_n^\t$, $\t=(1,n+1)$, $\FZERO=(x,(n+1)y)\in\QH_0^\t$, $\hx=-\fracp{1}{2}(y^2-x^{2n+2})\in\qh_{2(n+1)}^\t$, $\mu=dx^n\in\qh_n^\t$, with $d\in\mathbb{R}\setminus\llave{0}$. If $k\in\mathbb{N}$ then $\Ker\parent{\Opl_{k+2(n+1)}^{\co}}=\llave{0}$ if, and only if, $|d|\neq 1+\fracp{2(n+1)}{k}$. Moreover in this case $\Cor\parent{\Opl_{k+2(n+1)}^{\co}}=\llave{0}$.
\end{lemma}

\begin{proof}
Let us take $p\in\Delta_{k+n+2}$. Then, we can write \bean p(x,y) =
\alpha_0x^{k+n+2}+\alpha_1 x^{k+1}(y-x^{n+1})=(\alpha_0-\alpha_1)x^{k+n+2}+\alpha_1x^{k+1}y.\eean
Consider $\G_n=\F_n-\fracp{2(n+1)}{k+2(n+1)}dx^n\FZERO $, then
 \bean
\nabla p(x,y)\cdot\G_n & = & (k+n+2)(\alpha_0-\alpha_1) x^{k+n+1}\nabla x\cdot\G_n \\
&& +(k+1)\alpha_1x^{k}y\nabla x\cdot\G_n+\alpha_1x^{k+1}\nabla y\cdot\G_n. \eean
Taking into account that
\bean
\nabla x\cdot\G_n&=&y+\fracp{k}{k+2(n+1)}dx^{n+1},\\
\nabla y\cdot\G_n&=&(n+1)x^{2n+1}+\fracp{(n+1)k}{k+2(n+1)}x^ny,\\
y^2&=&x^{2n+2}-2\hx,
\eean
we have that
\bean
\nabla p(x,y)\cdot\G_n & = & \fracp{k(k+n+2)}{k+2(n+1)}\left[d+1+\fracp{2(n+1)}{k}\right]\alpha_0 x^{k+2(n+1)}\\ &&+\fracp{k(k+n+2)}{k+2(n+1)}\left[(1+\fracp{2(n+1)}{k})\alpha_0+(d-1-\fracp{2(n+1)}{k})\alpha_1\right]x^{k+n+1}(y-x^{n+1})\\
&&-2(k+1)\alpha_1 x^k\hx. \eean
In this way
\bean
\Opl_{k+2(n+1)}^{\co}\parent{p}&=&\fracp{k(k+n+2)}{k+2(n+1)}x^{k+n+1}\left[\left[d+1+\fracp{2(n+1)}{k}\right]\alpha_0 x^{n+1} \right.\\
&& \left. +\left[(1+\fracp{2(n+1)}{k})\alpha_0+(d-1-\fracp{2(n+1)}{k})\alpha_1\right](y-x^{n+1})\right],\eean
Hence choosing the bases $\Delta_{k+n+2}=<x^{k+n+2},x^{k+1}(y-x^{n+1})>$ and $\Delta_{k+2(n+1)}=<x^{k+2(n+1)},x^{k+n+1}(y-x^{n+1})>$ we obtain that the matrix of the operator $\Opl_{k+2(n+1)}^{\co}$ is similar to
$$\left(\begin{array}{cc}d+1+\fracp{2(n+1)}{k}&0\\ 1+\fracp{2(n+1)}{k}&d-1-\fracp{2(n+1)}{k}\end{array}\right)$$
from the structure of the above matrix the proof follows.
\end{proof}

As consequence of the Theorem \ref{teoFNlequivFn} we obtain the following orbital normal form of nilpotent vector with first quasi-homogeneous component non-conservative in the generic case.

\begin{theorem}\label{teFNnilpdisip} Consider $\F=\sum_{j\geq n}\F_j$, $\F_j\in\QH_j^\t$ with $\t=(1,n+1)$ and $\F_n=(y+dx^{n+1},(n+1)x^{2n+1}+(n+1)dx^ny)^T$ with $d\in\mathbb{R}\setminus\llave{0}$. If $|d|\neq 1+ \fracp{2(n+1)}{k}$ for all $k\in\mathbb{N}$ then $\F$ is orbitally equivalent to
$$\F_n+\sum_{k>n}\mu_k\FZERO,$$
where $\mu_k\in\Cor\parent{\Opl_{k}}$.
\end{theorem}

\begin{proof}
If $|d|\neq 1+ \fracp{2(n+1)}{k}$ for all $k\in\mathbb{N}$, by Lemmma \ref{lebasepkt} we have that $\Cor\parent{\Opl_{k+2(n+1)}^{\co}}=\Ker\parent{\Opl_{k+2(n+1)}^{\co}}=\llave{0}$. Applying now Theorem \ref{teoFNlequivFn} the result follows.
\end{proof}

\subsection{Orbital normal form for system (\ref{SistObjestu1}).}

In order to obtain an orbital normal form for system (\ref{SistObjestu1}), it is necessary to calculate $\Cor\parent{\Opl_k}$ with $k>n$, for this we need the following two technical lemmas that are referred to polynomial quasi-homogeneos vector fields.

\begin{lemma}\label{irred1}
Consider $\F_{n} = \X_{\hx} +  \mu\FZERO \in \QH_{n}^{\t}$
irreducible and $f \in \mathbb{C}[x,y]$ an invariant curve of  $\F_{n}$, irreducible. If $\ell_{n+k}(p) \in \left<f\right>$
then $p \in \left<f\right>$.
\end{lemma}
\begin{proof}
If $\Opl_{n+k}\parent{p}= 0$ then $p$ is a first integral of
$\dot{\x}=\F_{n}$. A first integral of $\F_{n}$ vanishes on any invariant curve of it, i.e., $p(\x)=0$ when
$f(\x)=0$. Therefore, by Hilbert`s Nullstellensatz $p\in rad\left<f\right>$.
Since $\left<f\right>$ is a prime ideal, then  $\left<f\right>=rad\left<f\right>$, in consequence
$p \in \left<f\right>$.

If $\Opl_{n+k}\parent{p}\neq 0$, let
$\nu\in \mathbb{C}[x,y]\setminus\llave{0}$ such that
$f\nu=\Opl_{r+k}\parent{p}$. Consider $\gamma(t)$, real or
complex, a solution curve of $\dot{\x}=\F_{n}(\x)$ which is a
parametrization of $f(\x)=0$. We assume that $\lim_{t\to
-\infty}\gamma(t)=0$, (the other case $\lim_{t\to
+\infty}\gamma(t)=0$ is proved in a similar way). Taking into account that $p(\zero)=0$
then
\bean p(\gamma(t))&=&p(\gamma(t))-p(\zero)=\int_{-\infty}^t \fracp{d p(\gamma(s))ds}{ds}=\int_{-\infty}^t \nabla_\x p \cdot\F_r(\gamma(s))ds \\
&=&\int_{-\infty}^t
\ell_{n+k}(p)(\gamma(s))ds=\int_{-\infty}^t f (\gamma(s))\nu(\gamma(s))ds=0\eean
 Recalling that $f(\x)=0$ is the union of orbits, we have that $p(\x)=0$ when $f(\x)=0$. Therefore, by Hilbert`s Nullstellensatz $p\in rad\left<f\right>$. Since $\left<f\right>$ is a prime ideal, then  $\left<f\right>=rad\left<f\right>$, in consequence
$p \in \left<f\right>$.
\end{proof}
\begin{remark} The hypothesis of the irreducibility of $\F_{n}$ is fundamental. For instance, if we consider $\F_{n}:=(y,bx^ny)^T\in\QH_n^{(1,n+1)}$, which is reducible with factor of reducibility $y$, we have that $\nabla x\cdot\F_{n}=y$ and nevertheless $x\notin\langle y \rangle$.
\end{remark}

\begin{lemma}\label{irredfm}
Consider $p\in\qh_k^\t$, $\F_{n} = \X_{\hx} +  \mu\FZERO \in \QH_{n}^{\t}$
irreducible and $f \in \qh_s^\t$ an irreducible
invariant curve of  $\F_{n}$ with $K_{n}$ its cofactor. If $\ell_{n+k}(p) \in \left<f^m\right>$ and  $\F_{n}+\fracp{jK_{n}}{k-js}\FZERO$ is irreducible for $j=1,\cdots,m-1$
then $p \in \left<f^m\right>$.
\end{lemma}
\begin{proof} Lemma \ref{irred1} proves the statement for $m=1$. We first consider the case $m=2$. If $\nabla p\cdot\F_{n}\in\langle f^2\rangle$ then $\nabla p\cdot\F_{n}\in\langle f\rangle$ and by \ref{irred1} we have that there exists $p_1\in\qh_{k-s}^\t$ such that $p=fp_1$, therefore
\bean
\nabla p\cdot\F_{n}&=&\nabla fp_1\cdot\F_{n}=p_1\nabla f\cdot\F_{n}+f\nabla p_1\cdot\F_{n}=p_1K_{n}f+f\nabla\p_1\cdot\F_{n}\\
&=& f\parent{\nabla p_1\cdot \fracp{K_{n}}{k-s}\FZERO+\nabla p_1\cdot\F_{n}}=f\nabla p_1\parent{\F_{n}+\fracp{K_{n}}{k-s}\FZERO}\in\langle f^2\rangle\eean
 Hence $\nabla p_1\cdot\parent{\F_{n}+\fracp{K_{n}}{k-s}\FZERO}\in\langle f\rangle$ and since $\F_{n}+\fracp{K_{n}}{k-s}\FZERO$ is irreducible applying Lemma \ref{irred1} we have that
$p_1\in\langle f\rangle$ and consequently $p\in\langle f^2\rangle$.

Consider now the case $m=3$. If $\nabla p\cdot\F_{n}\in\langle f^3\rangle$ then $\nabla p\cdot\F_{n}\in\langle f^2\rangle$ and by the previous paragraph we have that there exists $p_2\in\qh_{k-2s}^\t$ such that $p=f^2p_2$, therefore
\bean \nabla p\cdot\F_{n}&=&\nabla f^2p_2\cdot\F_{n}=p_2\nabla f^2\cdot\F_{n}+f^2\nabla p_2\cdot\F_{n}=2p_2K_{n}f^2+f^2\nabla\p_2\cdot\F_{n}\\
&=&f^2\parent{\nabla p_2\cdot \fracp{2K_{n}}{k-2s}\FZERO+\nabla p_2\cdot\F_{n}}=f^2\nabla p_2\parent{\F_{n}+\fracp{2K_{n}}{k-2s}\FZERO}\in\langle f^3\rangle\eean
Hence $\nabla p_2\cdot\parent{\F_{n}+\fracp{2K_{n}}{k-2s}\FZERO}\in\langle f\rangle$ and since $\F_{n}+\fracp{2K_{n}}{k-2s}\FZERO$ is irreducible applying Lemma \ref{irred1} we have that
$p_2\in\langle f\rangle$ and consequently $p\in\langle f^3\rangle$.
Reasoning by induction we get the result for $m\in\mathbb{N}$.
\end{proof}

Next results are referred to nilpotent vector fields whose first quasi-homogeneous component is integrable and non-conservative.

\begin{proposition}\label{proNablapenIM} Consider $\F_n$ the first quasi-homogeneous component of vector field (\ref{SistObjestu1}) and let $I_M\in\qh_M^\t$ be a primitive first integral of $\F_n$. Let $p\in\qh_k^\t$, if $\nabla p\cdot\F_n$ is multiple of $I_M$ then $p$ is zero or a not null multiple of $I_M$.
\end{proposition}

\begin{proof} We have that $k\geq M$ otherwise, if $k<M$ and $\Opl_{k+n}\parent{p}$ is a multiple of $I_M$ we obtain $\Opl_{k+n}\parent{p}=0$ which implies that $p$ is a first integral of $\F_n$ of degree less than $M$ which gives a contradiction since $I_M$ is a primitive first integral of $\F_n$.

The unique irreducible invariant curves of $\F_n$ are $f_1=(y-x^{n+1})\in\qh_{n+1}^\t$ and $f_2=(y+x^{n+1})\in\qh_{n+1}^\t$ with cofactors $K_1=(n+1)(d-1)x^n$ and $K_2=(n+1)(d+1)x^n$ respectively.

By Proposition \ref{pronilint} the polynomial $I_M=(y-x^{n+1})^{m_1}(y+x^{n+1})^{m_2}\in\qh_M^\t$, with $M=(n+1)(m_1+m_2)$ is a primitive first integral. It is sufficient to prove that if $\nabla p\cdot\F_n\in\langle f_i^{m_i}\rangle$ then $p\in\langle f_i^{m_i}\rangle$ for $i=1,2$.

\smallskip

We prove the case $i=1$, the case $i=2$ is analogous.

By Lemma \ref{irredfm} it suffices to prove that $\F_n+\fracp{jK_1}{k-j(n+1)}\FZERO$ is irreducible for $j=1,\cdots,m_1-1$. The vector field $\F_n+\fracp{jK_1}{k-j(n+1)}\FZERO=\F_n+\fracp{j(n+1)(d-1)}{k-j(n+1)}x^n\FZERO$ is irreducible if, and only if, $d\neq -1$ and $d\neq 1-\fracp{2j(n+1)}{k}$.

But $d\neq 1$ since $d=\fracp{m_1-m_2}{m_1+m_2}=1-\fracp{2m_2}{m_1+m_2}<1$ and $d\neq \fracp{2j(n+1)}{k}-1$ for $j=1,\cdots,m_1-1$ because
\bean d&=&\fracp{m_1-m_2}{m_1+m_2}=\fracp{2m_1}{m_1+m_2}-1=\fracp{2m_1(n+1)}{M}-1>\fracp{2j(n+1)}{M}-1\geq\fracp{2j(n+1)}{k}-1\eean
Therefore we have that $p\in\langle f_1^{m_1}\rangle$.
\end{proof}

\begin{lemma}\label{kerlk}
Assume that $I_M\in\qh_{M}^{\t}$, with $\t=(1,n+1)$, is a primitive first integral of $\F_n=(y+dx^{n+
1},(n+1)x^{2n+1}+d(n+1)x^ny)^T$, then

 $$Ker(\ell_k)= \left\{ \begin{array}{lcl}
             < I_M^l > &if& k-n = lM \\
             0  &if& k-n \neq lM
             \end{array}
   \right.$$
\end{lemma}
\begin{proof} It is sufficient to apply Proposition \ref{proNablapenIM}.
\end{proof}

Next statement establishes a cyclicity relation between the co-ranges of the operators $\ell_k$.

\begin{theorem}\label{ciclicidadl}
Assume that $I_M\in\qh_{M}^{\t}$, with $\t=(1,n+1)$, is a primitive first integral of $\F_n=(y+dx^{n+
1},(n+1)x^{2n+1}+d(n+1)x^ny)^T$, then we
can choose a complementary subspace to $\Range(\ell_{k+M})$, for $k\in\mathbb{N}$, $k\geq n$,
such that
$$\Cor(\ell_{k+M}) = I_M \Cor(\ell_{k}).$$
\end{theorem}

\begin{proof}
We will prove that $I_M \Cor(\ell_k) \subset \Cor(\ell_{k+M})$ or
equivalently that $I_M\Cor(\ell_{k})\cap
\Range(\ell_{k+M})=\{0\}$ by \emph{reductio ad absurdum}.
 Let $p\in\Cor\parent{\Opl_k}\setminus\llave{0}$ such that $pI_M\in\Range\parent{\Opl_{k+M}}$, then there exists $q\in\qh_{k+M-n}^\t\setminus\llave{0}$ (exists because $k\geq n$ and therefore $\qh_{k-n}^\t\neq\llave{0}$), such that $\Opl_{k+M}(q)=pI_M$, that is, $\Opl_{k+M}(q)$ is multiple of $I_M$. Applying Proposition \ref{proNablapenIM} we have that $q=\tilde{q}I_M$ with $\tilde{q}\in\qh_{k-n}^\t\setminus\llave{0}$ and consequently
\bean pI_M&=&\nabla q\cdot\F_n=\nabla\tilde{q}I_M \cdot\F_n=I_M\nabla\tilde{q}\cdot\F_n.\eean
Hence $p=\nabla\tilde{q}\cdot\F_n$, that is, $p\in\Range\parent{\Opl_k}\cap\Cor\parent{\Opl_k}$ which gives a contradiction.

We have shown that $I_M \Cor(\ell_k) \subset
\Cor(\ell_{k+M})$. Therefore it is enough to prove that
$\dim(I_M \Cor(\ell_{k}))=\dim(\Cor(\ell_{k+M}))$.

Since $\ell_{k}$ and $\ell_{k+M}$ are linear operators, we
get
\begin{equation}\label{dimCork}
\dim(\Cor(\ell_{k}))= \dim(\qh_{k}^\t) - \dim(\qh_{k-n}^\t) +
\dim(\Ker(\ell_{k})). \end{equation}
\begin{equation}\label{dimCork+r+t}
\dim(\Cor(\ell_{k+M}))= \dim(\qh_{k+M}^\t) -
\dim(\qh_{k+M-n}^\t) + \dim(\Ker(\ell_{k+M})). \end{equation}

From (\ref{dimCork}), (\ref{dimCork+r+t}), taking into account that $\qh_{k-n}^\t\neq\llave{0}$ for $k\geq n$ and applying \cite[Lemma 2.3]{AlgabaNonlinearity09} we obtain that
\[
\dim\parent{\qh_{k+M}^{\t}}-\dim\parent{\qh_{k+M-n}^{\t}}
=\dim\parent{\qh_{k}^{\t}}-\dim\parent{\qh_{k-n}^{\t}}.
\]
So, we get
\bean\dim(\Cor(\ell_{k+M}))= \dim(\Cor(\ell_{k}))-\dim(\Ker(\ell_{k}))+ \dim(\Ker(\ell_{k+M})).\eean

Using Lemma \ref{kerlk} we obtain
$\dim(\Cor(\ell_k))=\dim(\Cor(\ell_{k+M}))$ and this completes
the proof.
\end{proof}

Next result provides an orbital normal form of vector field (\ref{SistObjestu1}). This normal form depends on the first integral of the first quasi-homogeneous component of (\ref{SistObjestu1}) and it is a suitable normal form for the applications.

\begin{theorem}\label{NFhfsimples} System (\ref{SistObjestu1}) is orbital equivalent to
$$\dot{\x}=\F_n+\sum_{j=n+1}^{M+n-1}\eta_j^{(0)}\FZERO+\sum_{i=1}^{\infty}\sum_{j=n}^{M+n-1}\eta_j^{(i)}I_M^i\FZERO,$$
with $\eta_j^{(i)}\in\Cor(\ell_j)$.
\end{theorem}

\begin{proof}
As $|d|=\left|\fracp{m_1-m_2}{m_1+m_2}\right|<1<1+\fracp{2(n+1)}{k}$ for all $k\in\mathbb{N}$, applying Theorem \ref{teFNnilpdisip} we can assert that $\F$ is orbital equivalent to $\F_n+\sum_{j>n}\mu_j\FZERO$ with $\mu_j\in\Cor\parent{\Opl_j}$.
To finish the proof it is sufficient to apply Theorem \ref{ciclicidadl} for $j\geq M+n$.
\end{proof}

\noindent{\bf Remark:} In the particular case that $\F_n$ has a first integral, for calculating an orbital normal form of system (\ref{SistObjestu1}), we only
need the computation of a certain number of these co-ranges.
Concretely, $\Cor\parent{\Opl_j}$ from $j=n$ to $M+n-1$.

\section{Proofs of Theorems \ref{main} and \ref{main2}}

The proof of Theorem \ref{main} follows from the following proposition which characterizes the analytic integrability of a vector field through its orbital normal form.

\begin{proposition}\label{CNSintnilp} Let $\G=\sum_{j\geq r}\G_j$, $\G_j\in\QH_j^\t$, such that  $\G_r\in\QH_r^\t$ is one of the following vector fields.
\begin{description}
\item[a)] $\G_{r}=(y,x^{2n})^T\in\QH_{r}^\t$ with $\t=(2,2n+1)$, $r=2n-1$, $n\in\mathbb{N}$.
\item[b)] $\G_{r}=(y,\sigma x^{2n+1})^T\in\QH_r^\t$ with $\t=(1,n+1)$, $r=n$, $n\in\mathbb{N}$, $\sigma=\pm 1$.
\item[c)] $\G_{r}=(y+dx^{n+1},(n+1)x^{2n+1}+(n+1)dx^ny)^T\in\QH_r^\t$, $\t=(1,n+1)$, $r=n$, $n\in\mathbb{N}$, $d=\fracp{m_1-m_2}{m_1+m_2}\neq 0$ with $m_1,m_2\in\mathbb{N}$, coprimes.
\end{description}
$\G$ is analytically integrable if, and only if, $\G$ is orbitally equivalent to $\G_r$
\end{proposition}
\begin{proof} The sufficiency is trivial since $\G_{2n-1}$ of statement {\bf a)} and $\G_{n}$ of statement  {\bf b)} are polynomially integrable because are Hamiltonian vector fields and $\G_n$ of statement {\bf c)} is polynomially integrable by Theorem \ref{ci}.

\smallskip

Now we will see the necessity of the conditions. Let $\G=\sum_{j\geq r}\G_j$, $\G_j\in\QH_j^\t$ such that $\G_r$ is a quasi-homogeneous vector field of type given in some of statements {\bf a)}, {\bf b)} or {\bf c)}.
\begin{description}
\item[a)] If $\G_r=(y,x^{2n})\in\QH_{r}^\t$ with $\t=(2,2n+1)$ and $r=2n-1$ applying \cite[Proposition 13, Theorem 14]{Algaba_Garcia_Reyes_Fiv} we have that $\G_r+\cdots$ is orbitally equivalent to $\G_r+\sum_{j>r}\Cor\parent{\Opl_j}\FZERO$. Applying now \cite[Theorem 3.19]{AlgabaNonlinearity09} the result follows.
\item[b)] If $\G_r=(y,\sigma x^{2n+1})\in\QH_{r}^\t$ wit $\t=(1,n+1)$ and $r=n$ applying \cite[Proposition 15, Theorem 16]{Algaba_Garcia_Reyes_Fiv} we have that $\G_r+\cdots$ is orbitally equivalent to  $\G_r+\sum_{j>r}\Cor\parent{\Opl_j}\FZERO$. Applying now \cite[Theorem 3.19]{AlgabaNonlinearity09} the result follows.
\item[c)] If $\G_{r}=(y+dx^{n+1},(n+1)x^{2n+1}+(n+1)dx^ny)^T\in\QH_n^\t$, with $\t=(1,n+1)$, $n\in\mathbb{N}$, $r=n$ and $d=\fracp{m_1-m_2}{m_1+m_2}\neq 0$ with $m_1,m_2\in\mathbb{N}$, $m_1, m_2$ coprimes. Since $d=\fracp{m_1-m_2}{m_1+m_2}<1< 1+\fracp{2(n+1)}{k}$, applying Proposition \ref{teFNnilpdisip} we have that $\G$ is orbitally equivalent to $\widetilde{\G}:=\G_n+\sum_{j>n}\mu_j\FZERO$, with $\FZERO=(x,(n+1)y)^T$ and  $\mu_j\in\Cor\parent{\Opl_j}$. By Theorem \ref{ciclicidadl} we have that $I_M\Cor\parent{\Opl_j}\cap\Range\parent{\Opl_{M+j}}=\llave{0}$.

Assume that $\widetilde{\G}$ is integrable and not all the $\mu_j$ are zero. Let $N$ defined by
$N=\min\llave{j>n : \mu_j\not\equiv 0}$, by Theorem \ref{teointprimeranilp} a first integral of $\widetilde{\G}$ is of the form  $I=I_M+\sum_{j>M}I_j$ with $I_j\in\qh_j^\t$. Imposing
    the integrability condition we have
\bean 0&=&\parent{\nabla I\cdot\widetilde{\G}}_{N+M}=\nabla I_M\cdot\parent{\mu_N\FZERO}+\nabla I_{M+N-n}\cdot\G_n\\
&=&M\mu_NI_M+\Opl_{M+N}\parent{I_{M+N-n}}\eean
But this equation is incompatible since by Theorem \ref{ciclicidadl} $M\mu_NI_M\in\Cor\parent{\Opl_{M+N}}$ and $\Opl_{M+N}\parent{I_{M+N-n}}=-M\mu_NI_M\in\Range\parent{\Opl_{M+N}}$ which is a contradiction.
\end{description}
\end{proof}

\begin{remark}The proof of Theorem \ref{main} follows of taking into account that by Proposition \ref{pronilint}, there exist $\Phi_0\in\QH_0^\t$, with $\det\parent{D\Phi_0(\zero)}\neq 0$ such that $\parent{\Phi_0}_*\F=\G_r+\cdots$. By applying Proposition \ref{CNSintnilp} $\parent{\Phi_0}_*\F$ is analytically integrable if, and only if, $\parent{\Phi_0}_*\F$ is orbitally equivalent to $\parent{\Phi_0}_*\F_r$ or equivalently $\F$ is orbitally equivalent to $\F_r$.
\end{remark}

\begin{remark}The proof of Theorem \ref{main2}  follows from Theorem \ref{main} and applying \cite[Theorem 1.3]{Algaba09_like}.\end{remark}

\section{Example}

Consider the integrability problem of the following nilpotent system, whose first quasi-homogeneous component is non-conservative.
 \bea\label{FD12disipdoscomp}
 \vedo{\dot{x}}{\dot{y}}=\vedo{y-\fracp{1}{3}x^2}{2x^3-\fracp{2}{3}x y} +
 \vedo{a_1xy}{b_0 x^4 + b_2y^2}.
 \eea

The first quasi-homogeneous component of the vector field respect to the type $\t=(1,2)$ is
\[
\F_1=\vedo{y-\fracp{1}{3}x^2}{2x^3-\fracp{2}{3}xy}\in\QH_1^{(1,2)},
\]
where $\F_1=\X_{\hx}+\mu\FZERO$ with $\hx=-\fracp{1}{2}\left(y^2-x^4\right)$, $\mu=-\fracp{1}{3}x$ and $\FZERO=(x,2y)^T$. Hence the origin of system (\ref{FD12disipdoscomp}) is not monodromic. In fact is a saddle with two invariant curves $C_1=(y-x^2)+\cdots$ and $C_2=(y+x^2)+\cdots$. Moreover $\F_1$ is integrable and a primitive first integral is $I_6=(y-x^2)(y+x^2)^2$.
In order to study the integrability problem of family (\ref{FD12disipdoscomp}), the following lemma provides the formal normal form under equivalence of any perturbation of $\F_1$.

\begin{lemma} A formal normal form under equivalence of the system $\dot{\x}=\F_1+\cdots$, with $\F_1=(y-\fracp{1}{3}x^2, 2x^3-\fracp{2}{3}xy)^T$ is given by
\bea\label{FNF1} \vedo{\dot{x}}{\dot{y}}&=&\vedo{y-\fracp{1}{3}x^2}{2x^3-\fracp{2}{3}xy}+\alpha_2^{(0)}x^2\FZERO+\alpha_4^{(0)}\hx\FZERO\\
\nonumber&&+
\sum_{l\geq 1} \alpha_0^{(l)}I_6^l\FZERO+\alpha_1^{(l)}xI_6^l\FZERO+\alpha_2^{(l)}x^2I_6^l\FZERO+
\alpha_4^{(l)}\hx I_6^l\FZERO.
\eea
\end{lemma}
\begin{proof} Applying Theorem \ref{NFhfsimples} it is only necessary to compute $\Cor\parent{\Opl_j}$ for $n\leq j \leq M+n-1$, i.e., for $1\leq j\leq 6$.
\begin{itemize}
\item Case $j=1$. In this case
$\qh_{0}^\t=span\{1\}$ and
$\qh_{1}^\t=span\{x\}$. If we take $\mu_{0}=\alpha
\in \qh_{0}^t$ then, $\ell_1(\mu_{0})= 0$.
In consequence $\Cor(\ell_1)=span\{x\}$.
\item Case $j=2$. In this case
$\qh_{1}^\t=span\{x\}$ and
$\qh_{2}^\t=span\{x^2,xy\}$. If we take $\mu_{1}=\alpha
x\in \qh_{1}^t$ then, $\ell_2(\mu_{1})= \alpha (y-\fracp{1}{2}x^2)$.
In consequence $\Cor(\ell_2)=\{x^2\}$.
\item Case $j=3$. In this case $\qh_{2}^\t=span\{x^{2},
y\}$ and $\qh_{3}^\t=span\{x^{3},xy\}$. If we take
$\mu_{2}=\alpha x^{2}+\beta y \in \qh_{2}^t$ then,
$\ell_3(\mu_{2})=-\fracp{2}{3}(\alpha-3\beta)x^3+\fracp{2}{3}(3\alpha-\beta)xy$. Therefore we obtain
$\Cor(\ell_3)=\{0\}$.
\item Case $j=4$. In this case $\qh_{3}^\t=span\{x^{3},
xy\}$ and $\qh_{4}^\t=span\{x^{4},x^2y, \hx\}$. If we take
$\mu_{3}=\alpha x^{3}+\beta xy \in \qh_{3}^t$ then,
$\ell_4(\mu_{3})=(3\beta-\alpha)x^4+(3\alpha-\beta)x^2y-2\beta \hx$. Therefore we obtain
$\Cor(\ell_4)=span\{\hx\}$.
\item Case $j=5$. In this case $\qh_{4}^\t=span\{x^{4},
x^2y,\hx\}$ and $\qh_{5}^\t=span\{x^{5},x^3y, x\hx\}$. If we take
$\mu_{4}=\alpha x^{4}+\beta x^2y+\gamma \hx \in \qh_{4}^t$ then,
$\ell_5(\mu_{4})=-\fracp{4}{3}[(\alpha-3\beta)x^5+(\beta+3\alpha)x^3y+(3\beta+\gamma)x\hx]$. Therefore we obtain
$\Cor(\ell_5)=\{0\}$.
\item Case $j=6$. In this case $\qh_{5}^\t=span\{x^{5},
x^3y,x\hx\}$ and $\qh_{6}^\t=span\{x^{6},x^4y, x^2\hx, I_6\}$. If we take
$\mu_{5}=\alpha x^{5}+\beta x^3y+\gamma x\hx \in \qh_{5}^t$ then,
$\ell_6(\mu_{5})=-\fracp{5}{3}(\alpha-3\beta)x^6-\fracp{5}{3}(\beta-3\alpha)x^4y-\fracp{2}{3}(4\gamma+9\beta)x^2\hx-\fracp{1}{2}\gamma I_6$. Therefore we obtain
$\Cor(\ell_6)=span\{I_6\}$.
\end{itemize}
\end{proof}

\begin{theorem}
System (\ref{FD12disipdoscomp}) is analytically integrable if, and only if, the following
condition holds
$$a_{1}=b_0+b_2=0.$$
\end{theorem}

\begin{proof}
The proof consist in computing successively the
coefficients of the dissipative part of the normal form (\ref{FNF1}) and imposing
their vanishing because they prevent the integrability.
The normal form has the expression given in (\ref{FNF1}). It is easy to show that the first coefficient of the normal form is
$$
\alpha_{2}^{(0)} = \fracp{1}{336}(27b_2+27b_0+58a_1).$$
Imposing the vanish of this coefficient we obtain the first condition of integrability given by
$$b_0= -\fracp{58}{27}a_1-b_2.$$
With this condition, the second coefficient takes the form
\bean \alpha_4^{(0)}&=&-\fracp{2}{31104}a_1\left(1562(a_1+\fracp{3605}{3124}b_2)^2+\fracp{1511831}{6248}b_2^2\right)\eean

Now we consider two cases, the case {\bf (i)} $a_1=0$ and the case {\bf (ii)} $a_1\neq 0$.
\begin{description}
\item[i)]  In this case we have $b_0=-b_2$ and the following coefficients are
 $\alpha_0^{(1)}=\alpha_1^{(1)}=\alpha_2^{(1)}=\alpha_4^{(1)}=\alpha_0^{(2)}=\alpha_1^{(2)}=\alpha_2^{(2)}=0$. The vector field obtained is
 \bean \vedo{y-\fracp{1}{3}x^2}{2x^3-\fracp{2}{3}x y}+\vedo{0}{-b_2x^4+b_2y^2},\eean
and this system has the analytic first integral $I=(y-x^2)(y+x^2)^2\exp(-3b_2x)$.

\item[ii)] In the case $a_1\neq 0$ we have $\alpha_4^{(0)}\neq 0$ because  $\left(1562(a_1+\fracp{3605}{3124}b_2)^2+\fracp{1511831}{6248}b_2^2\right)>0$.
    Hence by Theorem \ref{CNSintnilp} statement {\bf c)}, the system is not integrable.
\end{description}
\end{proof}

\vspace{0.5truecm}

\noindent {\bf Acknowledgments.} The first and second authors are partially supported by a MINECO/FEDER
grant number MTM2014-56272-C2-2 and by the \emph{Consejer\'{\i}a de Educaci\'on y
Ciencia de la Junta de Andaluc\'{\i}a} (projects P12-FQM-1658, FQM-276). The
third author is partially supported by a MINECO/FEDER grant number MTM2017-84383-P
and by an AGAUR (Generalitat de Catalunya) grant number 2017SGR 1276.

\end{document}